\documentclass[11pt]{article}
\usepackage{amsmath}
\usepackage{graphicx}
\usepackage{amsfonts, amsmath, amsthm, amssymb, latexsym, amscd, eufrak}

\newtheorem{thm}{\bf Theorem}
\newtheorem{lem}{\bf Lemma}
\newtheorem{cor}{\bf Corollary}
\newtheorem{exam}{\bf Example}

\newenvironment{prf}{{\noindent\bf Proof}\rm}{\hfill{$\Box$}}

\begin{document}

\title{The stochastic queue core problem on a tree}
\author{
Mehdi\ Zaferanieh
       {\thanks
       {Department of Mathematics, Hakim Sabzevari University, Tovhid town, Sabzevar,
        Iran, email: m.zaferanieh@sttu.ac.ir }}
Jafar\ Fathali
      {\thanks
      {Department of Mathematics, Shahrood University of Technology, University Blvd.,
        Shahrood, Iran, email: fathali@shahroodut.ac.ir}}
 }
\date{}
\maketitle

\begin{abstract}
In this paper, an stochastic queue core problem on
a tree, which seeks to find a core in an M/G/1 operating environment is investigated.
Let $T=(V,E)$ be a tree, an stochastic queue core of $T$ is assumed to be a path $P$, for which the
summation of the weighted distances from all vertices to the path
as well as the average response time on the path is minimized.
Some general properties of the stochastic queue core problem on the tree are presented, while an algorithm
with $\max\{o(n^2 l),o(n^2 log^2(n))\}$ is provided to find the queue $l$-core on the tree.
\end{abstract}

\begin{quote}
\end{quote}
\section{Introduction}
The problem of finding the core  of a tree involves determining a path such that
the sum of the weighted distances from all vertices to this path is minimized.
Morgan and Slater \cite{MS80} developed an efficient linear time algorithm
to find a core of a tree. Becker et al. \cite{BCLSS02} and Alstrup et al. \cite{ALST01} presented an $O(n log^2 n)$
time algorithm to find a core on weighted trees whose length is at most $l$.
A branch-and-cut algorithm to find a core of general networks was presented by
Avella et al \cite{ABSV05}.

Zaferanieh and Fathali \cite{ZF12} investigated a
semi-obnoxious case on a tree in which some of the vertices are desirable, having positive weights,
while the others are undesirable, having negative weights. They proved whenever the
weight of tree is negative the core of the tree must be a single vertex, and when the
weight of tree is zero there exists a core that is a vertex.

The problem of finding a path-center on a tree (i.e. minimizing the maximal
distances to all vertices) was considered by Hedetniemi et al. \cite{HCH81} and
Slater \cite{S82}. The combination of finding a core and a path-center was investigated
by Averbakh and Berman \cite{AB99}. They presented an $O(n \log n)$ time algorithm for
the problem on a tree with $n$ vertices.

The stochastic queueing location problem mainly deals with $p-median$ models.
In the classical $p-median$ problem, the costumers would travel to the closest fixed
new facility to obtain the service required, while in the Stochastic Queue
Median $(SQM)$ problem. The service providers travel to the costumers
to provide the service required.

Kariv and Hakimi \cite{KH79} proved that the $p-median$ problem is NP-hard on general networks.
When the network is a tree they showed the problem can be solved in $O(p^2n^2)$ time. Tamir
\cite{T96} improved the time complexity to $O(pn^2)$. For the case $p = 1$, Goldman
\cite{G71} presented a linear time algorithm. For $p = 2$, an $O(n log n)$ time algorithm
was provided by Gavish and Sridhar \cite{GS95}.

Early work on $SQM$ problem was introduced by Berman et al.
\cite{BLC85}, which analyzed the $SQM$ model for the location of a single server
on a network operating as an $M/G/1$ queue. Chiu et al. \cite{CBL85} specialized
the Berman et al. \cite{BLC85} results for the case of a tree network.
The SQM was extended to finding the optimal location of a facility with $k$
servers by Batta and Berman \cite{BB89}. Finding the optimal
location of $p$ facilities with cooperation and without cooperation
between servers were investigated by Berman and Mandowsky
\cite{BM86} and Berman et al. \cite{BLP87}, resectively

Wang et al. \cite{WBR02} considered the problem of locating facilities that are
modelled as M/M/l queuing systems. Customers travel to the closest
facilities and a constraint is placed on the maximum expected waiting
times in all facilities. Berman and Drezner \cite{BD07} generalized this model
by allowing the facilities to host more than one server. To read about
other extensions of $SQM$ the reader is referred to Berman and Krass \cite{BK02}.

In this paper, the Stochastic Queue Core $(SQC)$ problem on a tree is considered.
The objective function is to find a path $P$ which minimizes the summation of
average response time as well as average distances from clients and manufacturing cost of the path.
When a service demand arises from a client, he/she should go to the closest vertex on the path.
Then in response to demand, the server travels (if available) along the path to provide the service.

The demands are assumed to arrive according to a homogenous Poisson process.
The model from a queueing standpoint is an M/G/1 queueing system.
It is also assumed the demands are places in a $FCFS$ (first
come first served) infinite capacity queue to await service.
Applications can be found in the design of high-speed communication networks.

In what follows, some notations and definitions are provided in Section 2.
Some properties and a sufficient condition to determine the $SQC$ are presented in Section 3.
In Section 4, the $l$-core problem is introduced while in Section 5 an algorithm with time complexity
$\max\{o(n^2 l),o(n^2 log^2(n))\}$ is given and some numerical examples are presented.

\section{Problem formulation}
Let $T=(V,E)$ be a tree, where $V$ is the set of vertices, $|V|=n$ and $E$ is the
set of edges. Demands for service occur on
the nodes of the tree. Each node $v_i$ generates an independent, Poisson
distributed stream of demands at rate of $\lambda_i$. Let $T_{u}$ be a subtree of $T$ including
all vertices $v\in T$ such that $d(u,v)\leq d(u_i,v)$ for all vertices while $u_i\in P$.
Table \ref{terms} contains a glossary of frequently used terms in this paper.

\begin{table}
\centering
\small
\begin{tabular}{ll}
$P$ & The path between two vertices\\
$\lambda=\sum_{i=1}^{n}\lambda_i$&system-wide average rate of arrival demands\\
$w_i=\frac{\lambda_i}{\lambda}$& fraction of calls originating from demand points $i$ \\
$w_{T_i}=\sum_{v_j \in T_i} w_j$ & the weight of branch $T_i$ \\
$p(i)$ & $i$th vertex of path $P$ \\
$d(v_{i},v_{j})$&  length of shortest path between vertices $v_{i}$ and $v_{j}$\\
$d(P,v)=\min_{u\in P}d(u,v)$& length of shortest path between path $P$ and vertex $v$\\
$G_i$& non-travel related service time at demand $i$\\
$s_i(P)$& service time of demand $i$\\
$TR_i(P)$&total response time of demand $i$\\
$Q_i(P)$& the time that demand $i$ wait in queue\\
$\overline{G}$&average non-travel service time\\
$\overline{S}(P)$&expected service time to a demand for server on the path $P$\\
$\overline{S}^{2}(P)$&second moment of service time to a demand for server on the path $P$\\
$\overline{TR}(P)$&expected total response time to a demand for server on the path $P$\\
$\overline{Q}(P)$&average queueing delay incurred by a demand for server on the path $P$\\
\end{tabular}
 \caption{The terms that used in the paper.}\label{terms}

\end{table}
Since the server travels along the path to response the requested services,
the position of the server at different times should be estimated. The probability that the server has been located
in a typical vertex $p\in P$ is indicated by $Prob_P(\hat{u}=p)$ and should be stated as follows:
$$Prob_P(\hat{u}=p)=\frac{\sum_{v_i\in T_{p}}\lambda_i}{\sum_{v_i\in T}\lambda_i}={\sum_{v_i\in T_{p}}w_i}=w_{T_p}.$$
Therefore, the mean distance from the server to a costumer which is located in vertex $p\in P$ is readily computed as follows:
$$\overline{d}_P(\hat{u},p)=\sum_{p_i\in P}Prob_P(\hat{u}=p(i))d(p(i),p).$$

When a service demand arises from a client (node), he/she would travel to
the closest vertex of the path to get his/her proposed service. Therefore, the client's mean distance
to reach the path is considered to be $\overline{d}_1(P)=\sum_{v_i\in T}w_id(P,v_i)$.
The average distance that the server travels to serve the clients is introduced by
$\overline{d}_2(P)=\sum_{v_i\in T}w_i\overline{d}_P(u,p(i))$,
where $p(i)$ is the closest vertex on the path $P$ to the vertex $v_i$.

If $vt$ is the {\em cruising speed}, then the travel time that clients reach
the path $P$ and the service time to serve them would be obtained by the following equalities, respectively:
$$\overline{T}_1(P)=\frac{1}{vt}\overline{d}_1(P),~~
\overline{T}_2(P)=\frac{1}{vt}\overline{d}_2(P).$$

Using these definitions the expected service time to response the requested services and
the total expected response time to reply all demands by the server on the path $P$ can be
calculated as follows, respectively: $$\overline{S}(P)=\overline{T}_2(P)+\overline{G}=\sum_{v_i\in T}w_is_i(P)+\overline{G},$$
$$\overline{TR}(P)=\beta( \overline{Q}(P)+\overline{T}_2(P))+\overline{G}+(1-\beta) \overline{T}_1(P),$$
where $0\leq \beta\leq 1$ and $s_i(P)$ is the total service time associated with a serviced demand from node $i$,
represented by the following:
$$s_i(P)=\frac{1}{vt}(\overline{d}_P(\hat{u},p(i))+G_i).$$
Here, it is assumed the Poisson arrival requests are served by a server which travels along the path $P$ and given by service distribution $s_i (P)$. In the $M/G/1$ case of queuing theory the inter arrival times and service times of consecutive costumers are assumed to be independent and identically distributed. In the proposed model, first the mean distance from the server to a costumer in vertex $p(i)$, i.e. $\overline{d}_P(\hat{u},p(i))$ is computed. Then the average distance that the server travels to serve the clients, i.e. $\overline{d}_2 (P)$ is calculated. Also the expected of service times on all vertices along the core are considered as the service times, see $\overline{S}(P)$ and $\overline{S^2}(P)$; therefore, the service times of consecutive costumers are independent and identically distributed.

If the manufacturing costs are also added to the problem, then
the objective function of $SQC$ problem becomes minimizing the following objective function:
$$F(P)=\alpha_1 |P|+\alpha_2\overline{TR}(P)=\alpha_1 |P|+
\alpha_2(\beta( \overline{Q}(P)+\overline{T}_2(P))+\overline{G}+(1-\beta) \overline{T}_1(P)),$$
where $\alpha_1\geq 0$ is the manufacturing costs of path, while $\alpha_2\geq 0$
is the price that should be paid for unit of time. Note that $\overline{G}$ is constant and can be removed,
while the other requested terms can be given as follows:
$$TR_i(P)=\beta(Q_i(P)+\frac{1}{vt}\overline{d}_P(\hat{u},p(i)))+\frac{(1-\beta)}{vt} d(P,v_i)+G_i,$$
$$\overline{S^2}(P)=\sum_{v_i\in T}w_is^2_i(P),$$

\begin{equation}\label{queue}
\overline{Q}(P) = \left\{\begin{array}{ccc}
                        \frac{\lambda \overline{S^2}(P)}{2(1- \lambda
                        \overline{S}(P))} &   if  &1- \lambda \overline{S}(P)>0 \\
                         \infty & & otherwise.
                       \end{array}\right.
\end{equation}

\section{properties}
In this section, properties of the objective function $F(P)$ is investigated.
In the rest of this paper, vertex $u$ is assumed to be adjacent to the path $P$.
The following lemma gives a recursion relation to find the first and second moment service time
on the path $P$ and consecutively on the path $P'=P\cup \{u\}$.

\begin{lem}\label{sip}
Let $P$ be a path on the tree $T$, $u$ be a vertex adjacent to the path $P$ and
$P'=P\cup \{u\}$ then
$$\overline{S}(P')=\overline{S}(P) + \frac{2}{vt}w_{T_u}(w(T)-w_{T_u})d(P,u).$$
$$\overline{S}^2(P')=\overline{S}^2(P)+\frac{2}{vt}w_{T_u}d(P,u)\overline{S}(P)+ \frac{1}{vt^2}w_{T_u}(w(T)-w_{T_u})w(T)d^2(P,u)$$
\end{lem}

\begin{prf}
\begin{figure} \centering
\includegraphics[width=9cm]{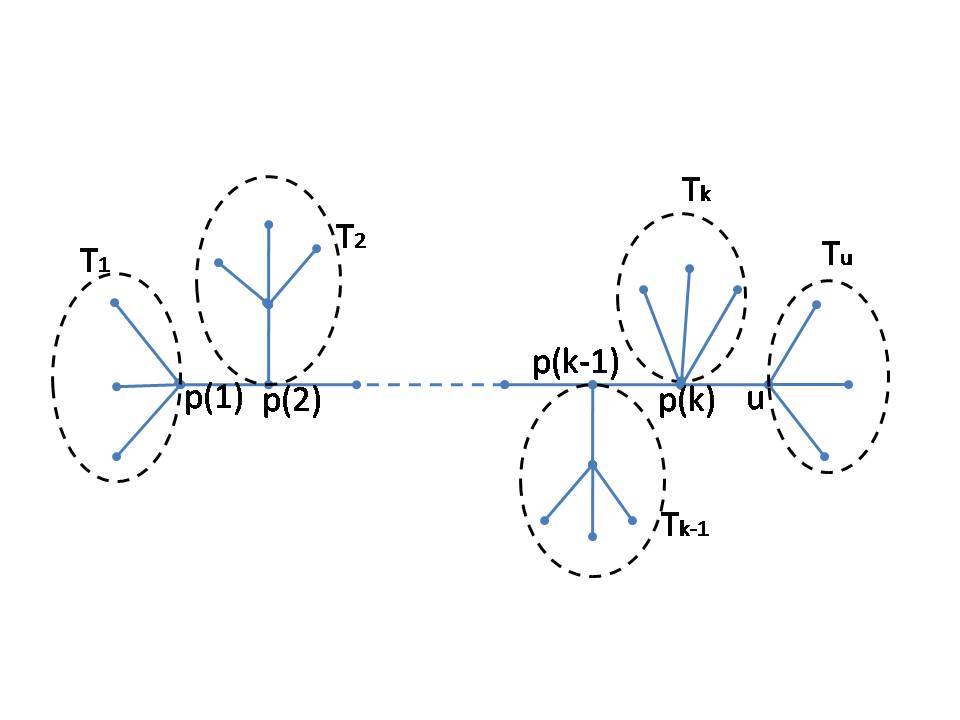}
\vspace{-1cm}
\caption{The branches of the path $P:p(1) - p(k)$}\label{tree2path}
\end{figure}
Consider Fig \ref{tree2path}, and let $T_{p(i)}$ be the branch associated to
the vertex $p(i)$, $i=1,...,k$ while $T_u$ be the branch associated to the vertex $u$.
The following equalities hold:
\begin{equation}\label{pext}
\begin{split}
& s_1(P')=s_1(P)+\frac{1}{vt}w_{T_u}d(P,u), \\
& s_2(P')=s_2(P)+\frac{1}{vt}w_{T_u}d(P,u), \\
 & \vdots \\
& s_k(P')=s_k(P)+\frac{1}{vt}w_{T_u}d(P,u),\\
& s_{k+1}(P') = s_k(P)+\frac{1}{vt}(w_{T_{p(1)}}+...+w_{T_{p(k)}})d(P,u).
\end{split}
\end{equation}
Multiplying both sides of the above equalities by $w_{T_{p(i)}}$
and adding them together yields the following equality:
\begin{equation*}
\begin{split}
\overline{S}(P')=\overline{S}(P) + \frac{2}{vt}w_{T_u}(w_{T_{p(1)}}+...+w_{T_{p(k)}})d(P,u) \\
= \overline{S}(P) + \frac{2}{vt}w_{T_u}(w(T)-w_{T_u})d(P,u).
\end{split}
\end{equation*}
Squaring both sides of the equalities \ref{pext}, multiplying them by $w_{T_{p(i)}}$
and adding the results together yields the following equality:
\begin{equation*}
\begin{split}
\overline{S}^2(P')=\overline{S}^2(P)+\frac{2}{vt}w_{T_u}d(P,u)\overline{S}(P) ~~~~~~~~~~~~~~\\ + \frac{1}{vt^2}w_{T_u}(w_{T_{p(1)}}+...+w_{T_{p(k)}})(w_{T_{p(1)}}+...+w_{T_{p(k)}}+w_{T_u})d^2(P,u) \\
= \overline{S}^2(P)+\frac{2}{vt}w_{T_u}d(P,u)\overline{S}(P)+ \frac{1}{vt^2}w_{T_u}(w(T)-w_{T_u})w(T)d^2(P,u).
\end{split}
\end{equation*}
\end{prf}

As a result, the queueing delay on the path $P'=P\cup \{u\}$ can significantly be computed, see Corollary \ref{qbar}.

\begin{cor}\label{qbar}
Let $P$ be a path on the tree $T$ while $P'=P\cup \{u\}$ and vertex $u$
is adjacent to the path $P$ then:

\begin{equation}
\begin{split}
\overline{Q}(P') = \frac{\lambda (\overline{S}^2(P) + \frac{2}{vt} w_{T_u}d(P,u) \overline{S}(P)+\frac{1}{vt^2}
w_{T_u}(w(T)-w_{T_u})w(T)d^2(P,u))}{2(1-\lambda (\overline{S}(P)+ \frac{2}{vt}w_{T_u}(w(T)-w_{T_u})d(P,u))}
\end{split}
\end{equation}
\end{cor}

As a result the following property is obtained from the definition
of $\overline{T}_1(P)$, Lemma \ref{qbar} and Corollary \ref{sip}.

\begin{cor}\label{pprim}
Let $P$ and $P'$ be two paths on the tree $T$ such that $P\subset P'$, then
the following equalities hold:
\begin{enumerate}
\item{$\overline{T}_1(P')\leq \overline{T}_1(P),~~\overline{T}_2(P)\leq\overline{T}_2(P')$}
\item{$\overline{S}(P)\leq\overline{S}(P')$, $\overline{S}^2(P)\leq\overline{S}^2(P')$}
\item{$\overline{Q}(P)\leq \overline{Q}(P')$}
\end{enumerate}
\end{cor}

Now we are ready to provide a condition to compare the objective function on the path $P$ and $P'=P\cup \{ u\}$.

\begin{lem}\label{pathpu}
Let $P'=P \cup \{u\}$, where $P$ is a path and $u$ is a vertex adjacent to the path $P$.
If $d(P,u)>0$ and
\begin{equation}\label{betterp}
\alpha_1 + \alpha_2w_{T_u}\left(\frac{2}{vt}\beta (w(T)-w_{T_u}) -\frac{1}{vt}+\frac{\beta}{vt}\right) >0
\end{equation}
then $F(P')>F(P).$
\end{lem}
\begin{prf}
The vertices of the path $P$ are assumed to be $p(1),...,p(k)$ while $d(P,u)=d(p(k),u)$, see Fig \ref{tree2path}.
Therefore, the vertices of path $P'$ are $p(1),p(2),...,p(k),p(k+1)=u$. The objective function
of the path $P'$ can be rewritten as follows:
\begin{equation*}
\begin{split}
& F(P')=\alpha_1|P'|+\alpha_2(\beta(\overline{Q}(P')+\overline{T}_2(P'))+\overline{G}+(1-\beta)\overline{T}_1(P')) \\
& =\alpha_1|P'|+\alpha_2\beta(\overline{Q}(P')+\overline{S}(P')-\overline{G})
+\alpha_2\overline{G}+\alpha_2(1-\beta)(\frac{1}{vt}\sum_{i=1}^{k+1}\sum_{v_j \in T_i}w_jd(p(i),v_j)). \\
\end{split}
\end{equation*}
Using Lemma \ref{sip}
\begin{equation*}
\begin{split}
& F(P')=\alpha_1(|P|+d(P,u))+\alpha_2\beta(\overline{Q}(P')+\overline{S}(P) + \frac{2}{vt}w_{T_u}(w(T)-w_{T_u})d(P,u)-\overline{G})\\
&+\alpha_2\overline{G}+\alpha_2(1-\beta)(\frac{1}{v}\sum_{i=1}^{k}\sum_{v_j \in T_i}w_jd(p(i),v_j)+\frac{1}{vt}\sum_{v_j \in T_u}w_jd(u,v_j)) \\
& =\alpha_1(|P|+d(P,u))+\alpha_2\beta(\overline{Q}(P')+\overline{S}(P) + \frac{2}{vt}w_{T_u}(w(T)-w_{T_u})d(P,u)-\overline{G})\\
&+\alpha_2\overline{G}+\alpha_2(1-\beta)\frac{1}{vt}(\sum_{i=1}^{k}\sum_{v_j \in T_i}w_jd(p(i),v_j)+\sum_{v_j \in T_u}w_j(d(p(k),v_j))-w_{T_u}d(P,u)) \\
&=\alpha_1|P|+\alpha_2\beta(\overline{Q}(P')+\overline{S}(P)-\overline{G})
+\alpha_2\overline{G}+\alpha_2(1-\beta)(\frac{1}{vt}\overline{T}_1(P)-w_{T_u}d(P,u))\\
&+ \alpha_2\frac{2\beta}{vt}w_{T_u}(w(T)-w_{T_u})d(P,u)+\alpha_1 d(P,u)\\
\end{split}
\end{equation*}
therefore
\begin{equation}
\begin{split}
&F(P')=F(P)+\alpha_1 d(P,u)+\alpha_2\beta (Q(P')-Q(P)) + \\
&\frac{2\alpha_2\beta}{vt} w_{T_u}(w(T)-w_{T_u})d(P,u) -(1-\beta)\frac{\alpha_2}{vt}w_{T_u}d(P,u)\\
&=F(P)+(\alpha_1+\frac{2\alpha_2\beta}{vt} w_{T_u}(w(T)-w_{T_u}) -(1-\beta)\frac{\alpha_2}{vt}(w_{T_u}))d(P,u)+\alpha_2\beta (Q(P')-Q(P))
\end{split}
\end{equation}
Since by Corollary \ref{pprim}, $Q(P')>Q(P)$ and $d(P,u)>0$ then if
$$\alpha_1 +\alpha_2 w_{T_u}\left(\frac{2}{vt}\beta (w(T)-w_{T_u}) -\frac{1}{vt}+ \frac{\beta}{vt}\right)>0$$
we conclude $F(P')>F(P)$ which completes the proof.
\end{prf}

Note that the condition \ref{betterp} in Lemma \ref{pathpu} is a sufficient
but not necessary condition, see Example \ref{exam1}.

\begin{exam}\label{exam1}
Consider the tree shown in Fig. \ref{exam-tree}. The numbers
written on edges indicate their lengths while demand rates, weights and service times of
vertices are presented in Table \ref{tablereq}. Table \ref{pathes} contains some paths and
their corresponding $\overline{T}_1$, $\overline{T}_2$, $\overline{Q}$. The head column objective
function includes $\alpha_1$ and $\beta$ with different values. The other variables are cruising speed $vt=1$, $\alpha_2=1$
$\lambda=0.34$, $\overline{G}=0.01$ and $w(T)=1$.

\begin{figure}[h]
  \centering
\includegraphics[width=11cm]{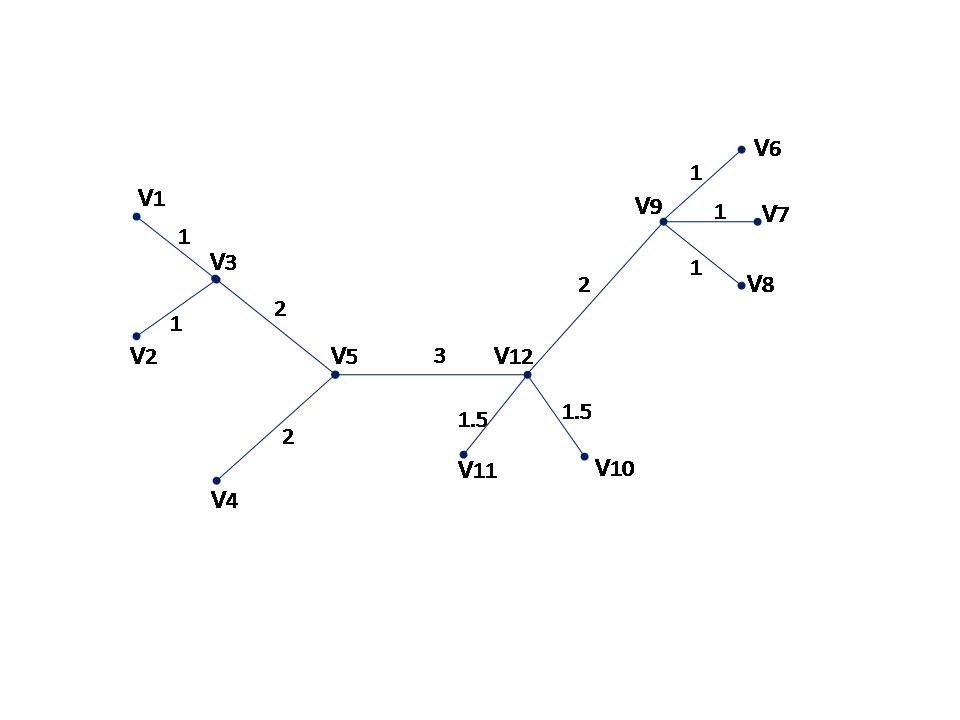}
\vspace{-1.5cm}
\caption{Tree for the Example \ref{exam1}. }\label{exam-tree}
\end{figure}

\begin{table}[h]
\centering
\tiny
\begin{tabular}{c|cccccccccccc}
vertices&$v_{1}$&$v_{2}$&$v_{3}$&$v_{4}$&$v_{5}$&$v_{6}$&$v_{7}$&$v_{8}$&$v_{9}$&$v_{10}$&$v_{11}$&$v_{12}$\\
\hline
$\lambda_i$&0.02&0.03&0.01&0.04&0.03&0.03&0.05&0.02&0.03&0.05&0.02&0.01\\
$w_i=\frac{\lambda_i}{\lambda}$&0.0588&0.0882&0.0294&0.1176&0.0882&0.0882&0.1471&0.0588&0.0882&0.1471&0.0588&0.0294\\
$G_i$&0.01&0.01&0.01&0.01&0.01&0.01&0.01&0.01&0.01&0.01&0.01&0.01\\
\end{tabular}
 \caption{The  demand rates and service times of vertices in
Example \ref{exam1}.}\label{tablereq}
\end{table}

At first consider two paths $P=\{v_5,v_{12}\}$ and $P'=P \cup \{v_9\}$ when $\alpha_1=\beta=0.1$.
Then $w_{T_{9}}=0.3824$ and $$\alpha_1 +\alpha_2 w_{T_{9}}(\frac{2}{vt}\beta (w(T)-w_{T_{9}}) -\frac{1}{vt}+\frac{\beta}{vt})=0.1-0.2969<0.$$
Despite the fact that condition (\ref{betterp}) of Lemma \ref{pathpu} is not satisfied for these two paths,
the inequality $F(P)<F(P')$ holds; therefore, this condition is not necessary.

Now consider two paths $P=\{v_5,v_{12}\}$ and $P'=P \cup \{v_9\}$ for the case $\alpha_1=0.5$ and $\beta=0.1$.
In this case the condition of Lemma \ref{pathpu} holds,
$$\alpha_1 + \alpha_2w_{T_{9}}(\frac{2}{vt}\beta (w(T)-w_{T_{9}}) -\frac{1}{vt}+\frac{\beta}{vt})=0.5-0.2969>0$$ and  also $F(P)<F(P')$.

Note that in the case $\alpha_1=\beta=0$ the problem is equivalent to the path-median problem.

\begin{table}
\centering
\begin{tabular}{|cccc|ccc|}
\hline
& & & &\multicolumn{3}{|c|}{ $objective~ function$} \\
\hline
& & & &$\alpha_1=0$&$\alpha_1=0.1$&$\alpha_1=0.5$\\
path&$\overline{T}_1$&$\overline{T}_2$&$\overline{Q}$&$\beta=0$&$\beta=0.1$&$\beta=0.1$\\
\hline
$\{v_1\}$&6.1324&0.0000&0.000017058&6.1424&5.5291&5.5291\\
$\{v_5\}$&3.9559&0.0000&0.000017058&3.9659&3.5703&3.5703\\
$\{v_{12}\}$&3.2500&0.0000&0.000017058&3.2600&2.9350&2.9350\\
$\{v_9\}$&3.7206&0.0000&0.000017058&3.7306&3.3585&3.3585\\
$\{v_5,v_{12}\}$&2.1029&1.4170&0.7112&2.1129&2.4155&3.6155\\
$\{v_{12},v_9\}$&2.4853&0.9446&0.2425&2.4953&2.5655&3.3655\\
$\{v_8,v_{9},v_6\}$&3.5733&0.2716&0.0319&3.5835&3.4565&4.2565\\
$\{v_1,v_{3},v_5\}$&3.5441&0.6920&0.1987&3.5541&3.5888&4.7888\\
$\{v_5,v_{12},v_{9}\}$&1.3382&2.3616&5.0013&1.3482&2.4507&4.4507\\
$\{v_3,v_5,v_{12}\}$&1.7500&1.9983&2.4618&1.7600&2.5310&4.5310\\
$\{v_3,v_5,v_{12},v_{9}\}$&0.9853&2.9429&$\infty$&0.9953&$\infty$&$\infty$\\
\hline
\end{tabular}
 \caption{The  objective functions of some paths in
Example \ref{exam1}.}\label{pathes}
\end{table}
\end{exam}

\section{Queue l-core problem}
In this section, the queue $l$-core problem is investigated in which
finding a path $P$ with fixed length $l$ that should be shifted
between two end vertices is the main problem. At first the variation of function
$s_i(P)=\frac{1}{vt}(\overline{d}_{P}(\hat{u},p(i))+G_i)$ is verified when the
position of path $P$ is changed between two end vertices $v_1$ and $v_2$.
For simplicity, in the rest of this paper it's assumed $G_i=0,~\alpha_1=0$ and $\alpha_2=1$.

\begin{lem}
Let $\overline{v}_1$ and $\overline{v}_2$ be two end vertices, and
$P$ be a path with fixed length $l$ that is shifted between
two vertices $\overline{v}_1$ and $\overline{v}_2$ by step size $d$. Then
\begin{equation}\label{ingred}ý
ý{s^{new}_i}(P) = \left\{\begin{array}{ccc}ý
                        ýs_i(P)-\frac{1}{vt}(w_{T_{p(2)}}+...+w_{T_{p(k)}})d &   if  & i=1 \\ý
                        ýs_{i+1}(P)-\frac{1}{vt}(w_{T_{p(1)}}-w_{T_{p(k+1)}})d & if & i=2,...,k-1 \\ý
                        ýs_i(P)+\frac{1}{vt}(w_{T_{p(2)}}+...+w_{T_{p(k)}})d & if & i=ký
                       ý\end{array}\right.
\end{equation}
\end{lem}
\begin{proof}
 The proof is straightforward, if we consider a path in two
consecutive cases, where the path $P$ is shifted by step size $d$, then the above relations hold.
\end{proof}
\begin{cor}
Let $\overline{v}_1$ and $\overline{v}_2$ be two end point vertices.
If the path $P$ is shifted between these two vertices then $\overline{S}(P)$ and
$\overline{S}^{2}(P)$ can be iteratively computed.
\end{cor}
\begin{lem}\label{lem1}
Let $\overline{v}_1$ and $\overline{v}_2$ be two end vertices
and also let $P$ be a path with fixed length which is shifted along the path
connecting these two vertices. Then $\overline{S}(P)$ is an increasing
decreasing function.
\end{lem}
\begin{proof}
\begin{figure}
  \centering
\includegraphics[width=9cm]{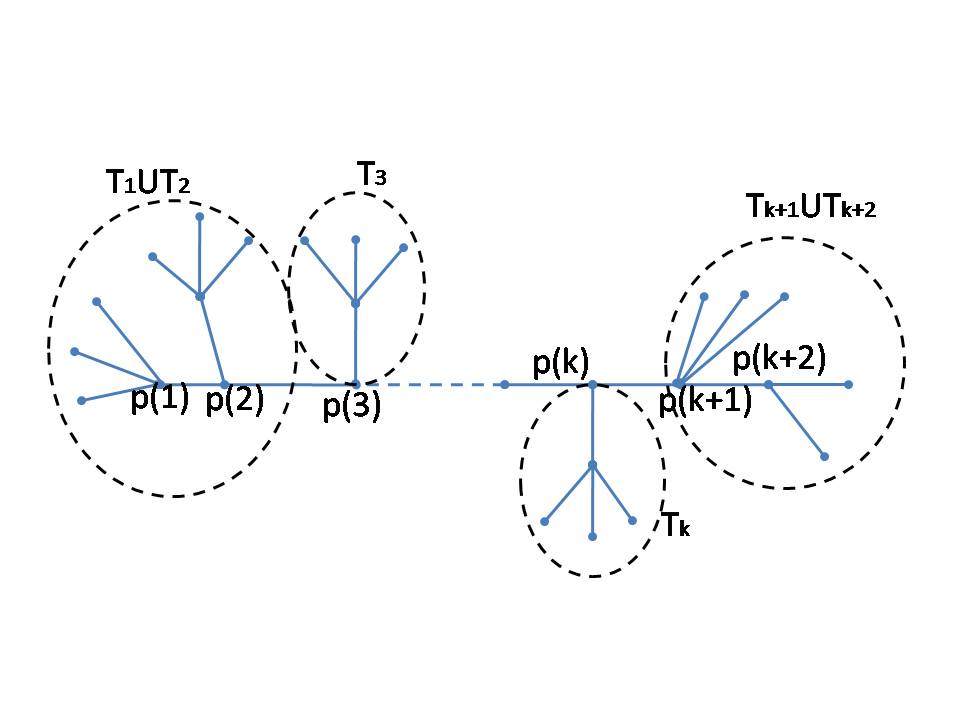}
\vspace{-1cm}
\caption{Tree for the Lemma \ref{lem1}. }\label{fixedpath}
\end{figure}
Consider the path $P:p(1)-p(k)$ as depicted in Fig \ref{fixedpath}.
 When the path $P$ is shifted from left to right three consecutive cases should be considered.
In first case, the weight of branch corresponding to vertex $p(k)$ is $w_{T_{p(k)}}+w_{T_{p(k+1)}}+w_{T_p(k+2)}$.
In second one, the weight of branch corresponding to vertex $p(2)$ is $w_{T_{p(1)}}+w_{T_{p(2)}}$ and the weight
of branches corresponding to the vertices $p(k)$ and $p(k+1)$ are $w_{T_{p(k)}}$ and $w_{T_{p(k+1)}}+w_{T_{p(k+2)}}$,
respectively. In third one, the weight of branches corresponding to vertices $p(3)$ and $p(k+2)$ are
$w_{T_{p(1)}}+w_{T_{p(2)}}+w_{T_{p(3)}}$ and $w_{T_{p(k+2)}}$. The weight of branches corresponding to all other vertices
in these three cases remain unchanged. The mean service time functions corresponding to
these three cases are represented as $\overline{S}^{1}(P)$, $\overline{S}^{2}(P)$ and $\overline{S}^{3}(P)$, respectively.

\begin{equation}\label{f1f2f3}
\begin{split}
\overline{S}^{1}(P) &= \frac{2}{vt} [w_{T_{p(1)}}\sum_{i=2}^{p(k-1)} w_{T_{p(i)}}d(p(i),p(1))+w_{T_{p(2)}}\sum_{i=3}^{k-1}w_{T_{p(i)}}d(p(i),p(2))+...\\
&+w_{T_{p(k-2)}}\sum_{i=k-1}^{k-1}w_{T_{p(i)}}d(p(i),p(k-2)) \\
&+(w_{T_{p(k)}}+w_{T_{p(k+1)}}+w_{T_{p(k+2)}})\sum_{i=1}^{k-1}w_{T_{p(i)}}d(p(i),p(k+1))]\\
&\overline{S}^{2}(P)=\frac{2}{vt}[(w_{T_{p(1)}}+ w_{T_{p(2)}})\sum_{i=3}^{k}w_{T_{p(i)}}d(p(i),p(2))+w_{T_{p(3)}}\sum_{i=4}^{k}w_{T_{p(i)}}d(p(i),p(3))+...\\
&+w_{T_{p(k-2)}}\sum_{i=k-1}^{k}w_{T_{p(i)}}d(p(i),p(k-2))
+w_{T_{p(k-1)}}\sum_{i=k}^{k}w_{T_{p(i)}}d(p(i),p(k-1))\\
&+(w_{T_{p(k+1)}}+ w_{T_{p(k+2)}})\sum_{i=2}^{k}w_{T_{p(i)}}d(p(i),p(k+1))]
\end{split}
\end{equation}

Since the length of path $P$ is fixed value $l$ then in foregoing equations $\overline{S}^{1}(P)$
and $\overline{S}^{2}(P)$ we have $d(p(1),p(k))=d(p(2),p(k+1))$; therefore, by subtracting them the following equation is given:
\begin{equation}
\begin{split}
\overline{S}^{1}(P)-\overline{S}^{2}(P)= & \frac{2d(p(1),p(k))}{vt}[w_{T_{p(1)}}(w_{T_{p(2)}}+...\\
&+w_{T_{p(k)}})-(w_{T_{p(k+1)}}+w_{T_{p(k+2)}})(w_{T_{p(2)}}+...+w_{T_{p(k)}})].
\end{split}
\end{equation}

If it is supposed $\overline{S}^{1}(P)-\overline{S}^{2}(P)>0$ then the inequality
$w_{T_{p(1)}}>w_{T_{p(k+1)}}+w_{T_{p(k+2)}}$ is attained. Moreover, applying the following conclusion:
$$w_{T_{p(1)}}>w_{T_{p(k+1)}}+w_{T_{p(k+2)}} \Rightarrow  w_{T_{p(1)}}+w_{T_{p(2)}}>w_{T_{p(k+2)}},$$
the inequality $\overline{S}^{3}(P)-\overline{S}^{2}(P)>0$ holds.

\end{proof}
\begin{cor}
Let the path $P$ be shifted between two end vertices $\overline{v}_1$ and $\overline{v}_2$ then the
mean service time function $\overline{S}(P)$ is minimized either at $\overline{v}_1$ or $\overline{v}_2$.
\end{cor}
It seems the mean service time function $\overline{S}(P)$ is concave, but the
following counter-example shows that this derivation is not true in general.

\begin{figure}
  \centering
\includegraphics[width=9cm]{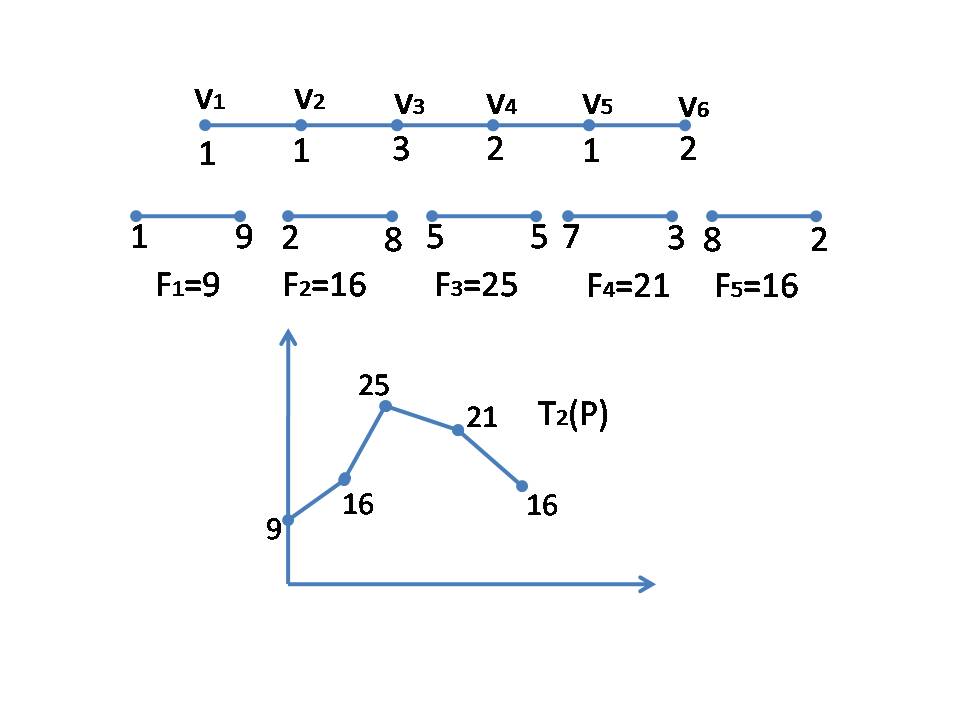}
\vspace{-1cm}
\caption{Counter Example \ref{exam1}. }\label{counterexam}
\end{figure}

\begin{exam}
Consider the figure depicted in Fig \ref{counterexam} The length of all edges are equal to one and
the weights of vertices are written next to the vertices. The results are illustrated
in Fig \ref{counterexam} which shows the function $\overline{S}(P)$, is increasing decreasing (but not concave).
The function $\overline{S}(P)$ takes its minimum value at vertex $v_1$.
\end{exam}

\section{Algorithm l-core}
In this section, an algorithm to find the stochastic queue l-core is provided.
The basic idea is based on finding the best $l$-core, passing
through a vertex $q$. In order to find the best path $P^{*}$
passing through the vertex $q$, it's considered the tree rooted in $q$ and use a bottom up iterative algorithm.
In the algorithm, some notations and definitions are required, which are defined below.

Let $q$ be the root of the tree $T$
while $child(q)$ be the set of its children. The parent of an arbitrary vertex $v$
is denoted by $par(v)$ while the subtree rooted at vertex $v$ is shown by $T_v$.
In the algorithm, the notation $(v,k,v_1,...,v_h)$ is used
as a path whose fixed length is $k$, started from the vertex $v$,
pass through vertices $v_1,...,v_{h-1}$ and terminated either on the edge $e_{v_{h-1},v_{h}}$ or on the vertex $v_h$.
Two paths $P_1$ and $P_2$ are called adjacent, if their vertices sets
are $(v_1,v_2,v_3,...,v_{h-1},v_{h})$ and $(v_2,v_3,...,v_{h-1},v_h,v_{h+1})$, respectively.

\textbf{Algorithm[SQLC]}\label{sqlc}
\begin{enumerate}
\item{Let the tree $T$ be rooted at the vertex $q$ while the path $P$ whose fixed length is $l$ pass
through the vertex $q$.}
\item{Let the children set of each vertex $v$ be $child(v)=\{v_1,...,v_c\}$.}
\item{If $v$ is a leaf, set $sum_{d}(v)=0$ and $\overline{w}_{down}(v)=w(v)$.}
\item{For every vertex $v$ of the tree $T$ compute the following values:
\begin{enumerate}
\item{$\overline{w}_{down}(v)=\sum_{i\in T_v}w_i=\sum_{v_i \in child(v)}\overline{w}_{down}(v_i)$.}
\item{$sum_d(v)=\sum_{v_i \in child(v)}\{sum_d(v_i)+d(v,v_i)\overline{w}_{down}(v_i)\}$}
\item{$sum^*_{d}(v)=sum_d(v)+\overline{w}_{down}(v)d(v,par(v))$}
\end{enumerate}}
\item{For each vertex $v$ of tree $T$ and for $k=0,1,...,l$ do the following:}
\begin{enumerate}
\item{If $k=0$ then set $down_c(v,k)=sum_d(v)$.}
\item{Else for every path $(v,k,v_1,...,v_h)$ compute $$down_c(v,k,v_1,...,v_h)=down_c(v,d(v,v_{h-1}),v_1,...,v_{h-1})
-\overline{w}_d(v_h)(k-d(v,v_{h-1}))$$} and
\begin{enumerate}
\item{If $k<d(v,par(v))$ then $down_c^{*}(v,k,v_1,...,v_h)=sum^{*}_{d}(v)-\overline{w}_{d}(v)(k-d(v,par(v)))$}
\item{Else $down_c^{*}(v,k,v_1,...,v_{h})=down_c(v_1,k-d(v,par(v)),v_2,...,v_h)$}
\end{enumerate}
\end{enumerate}
\item{For all end vertices, as $v_1$ and $v_2$, where $d(v_1,v_2)\geq l$ do the following:}
\begin{enumerate}
\item{Let $\{P_1,...,P_k\}$ be the set of all paths between $v_1$ and $v_2$ with length $l$ passing through $q$ those are
found by the previous step.}
\item{Let $P_{i}$ and $P_{i+1}$ for $i=1,...,k-1$, be two adjacent paths.}
\item{Let the path $P_i$ be represented as, $P_{i}:v_{i1}-v_{ih_i}$ for $i=1,...,k.$}
\begin{enumerate}
\item {Set $\overline{S}(P_1)=\overline{S}^2(P_1)=0$}
\item{For $j=1$ to $h_1$ do the following}
\begin{enumerate}
\item{$\overline{S}(P_1)=\overline{S}(P_1)+\frac{2}{vt}w_{T_{v_{1j}}}(w(T)-w_{T_{v_{1j}}})d(v_{1j-1},v_{1j}).$}
\item{$ \overline{S}^2(P_1) = \overline{S}^2(P_1)+\frac{2}{vt}w_{T_{v_{1j}}}d(v_{1j-1},v_{1j})\overline{S}(P_1)+
\frac{1}{{vt}^2}w_{T_{v_{1j}}}(w(T)-w_{T_{v_{1j}}})w(T)d^2(v_{1j-1},v_{1j}).$}
\end{enumerate}
\end{enumerate}
\item{For $i=2$ to $k$, Compute $\overline{S}(P_i)$ and $\overline{S}^{2}(P_i)$ as follows:}
\begin{enumerate}
\item{$\overline{S}(P_{i+1})=\overline{S}(P_i)+\frac{1}{vt}(w_{T_{v_{i1}}}-w_{T_{v_{i+1h_{i+1}}}})(w_{T_{v_{i2}}}+...+w_{T_{v_{ih_{i}}}})d(v_{i1},v_{ih_{i}}).$}
\item{Compute $\overline{S}^{2}(P_{i+1})$ using equations $\ref{ingred}$.}
\end{enumerate}
\end{enumerate}
\item{For each path $P_i$, $i=1,...,k$ with length $l$ set}
\begin{enumerate}
\item{$\overline{Q}(P_i)=\frac{\lambda\overline{S}^{2}(P_i)}{2(1-\lambda\overline{S}(P_i))}.$}
\item{$\overline{F}(P)=\overline{Q}(P_i)+\overline{TR}(P_i)$.}
\end{enumerate}
\item{Returns the path $P_h$, where $F(P_h)=\min\{F(P_i), i=1,...,k\}$, as the optimal solution.}
\end{enumerate}

\begin{thm}
The algorithm $SQLC$ finds the stochastic queue $l$-core of the tree $T$ in $\max\{O(n^2l), O(n^2log^2n)\}$ time.
\end{thm}

\begin{proof}
Step 4 can be done in $O(n)$ time. The main term in step 5 is computing $down_c$ which should be computed for every path with length $k$. The number of these paths that start from a vertex is at most $O(n)$. Since step 5 runs for $k=0,...,l$ then this step can be done in $O(n^2l)$ time. In step 6 we consider every two end vertices such as $v_1$ and $v_2$, that $d(v_1,v_2)\geq l$. The number of these pairs is $O(n^2)$. Then we should compute $\overline{S}(P)$ and $\overline{S}^{2}(P)$ for every path $P$ with length $l$ that lies between two end vertices and passes through the root $q$. Since these pathes should be pass through the root, number of them for each pair of end vertices is at most $O(h)$ where $h$ is the height of the tree. As Alstrup et al. \cite{ALST01} have been shown the tree can be compressed so that $h=O(log n)$. The time needed for computing $\overline{S}(P_{i+1})$ and $\overline{S}^{2}(P_{i+1})$ in part (d) of step 6 is $O(h)$. Therefore step 6 can be done in $O(n^2log^2n)$ time. Step 7 should be run for every path with length $l$ that have been found in the previous steps. Therefore the time complexity of this step dose not exceed of those steps 5 and 6.
\end{proof}

\subsection{Numerical examples}
\begin{figure}
  \centering
  \vspace{-2cm}
\includegraphics[width=14cm]{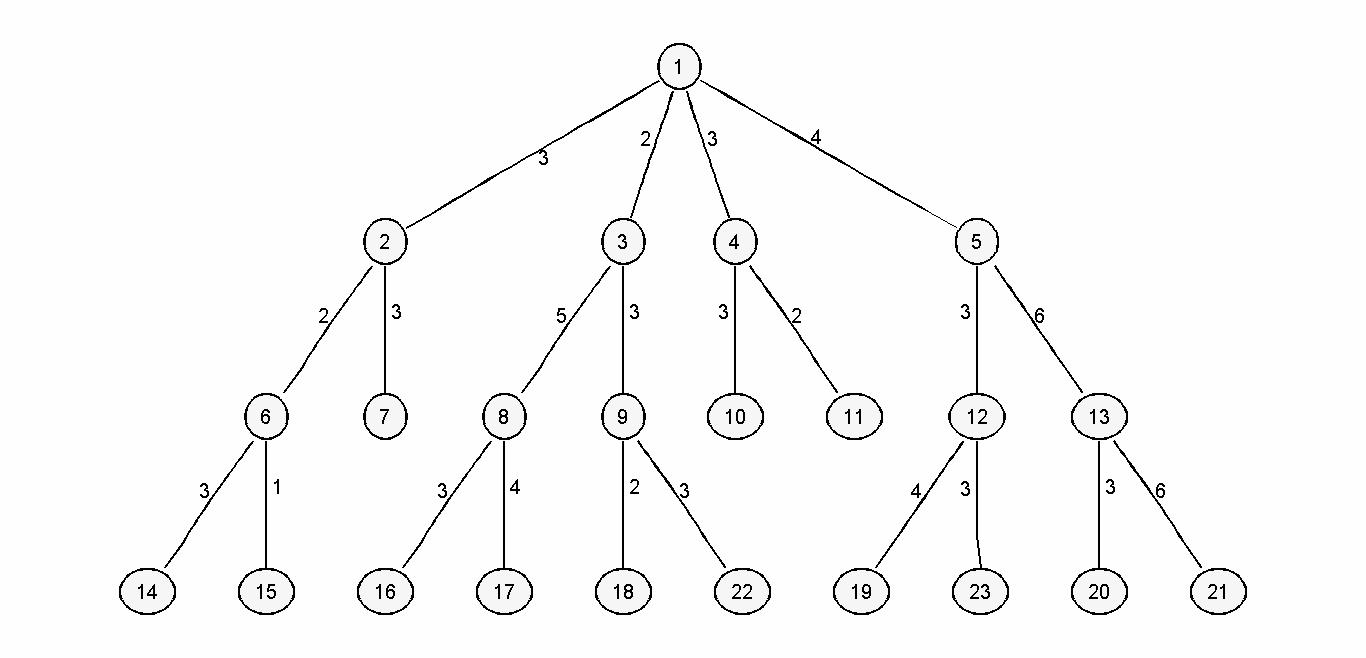}
\caption{An example. }\label{exam-core}
\end{figure}

In this section, first the stochastic queue $l$-core of the tree represented in
Fig \ref{exam-core} for l=1,2,...,6 is verified. The algorithm was written in MATLAB and run on a PC
with a Pentium IV processor with 1 GB of RAM and CPU 2 GHz. We consider the root of the tree is vertex $v_1$.

In Fig \ref{exam-core}, the numbers written next to the edge are represented
their lengths.
The weights of vertices are equal and considered to be $w_i=\frac{1}{num}$,
where $num$ is the number of vertices.
\begin{table}[h]
\centering
\begin{tabular}{|c||ccccccccc|}
\hline
&\multicolumn{9}{|c|}{$|core|$} \\ \hline
$\lambda$ & $1$ & $2$ & $3$ & $4$ & $5$ & $6$ & $7$ & $8$ & $9$\\ \hline
$0.1$ & $6.80$     & $6.54$     & $6.33$     & $6.16$     & $6.05$     & $6.02$     & $6.03$     & $6.07$     & $6.13$   \\
$0.2$ & $6.82$     & $6.62$     & $6.55$     & $6.60$     & $6.70$     & $6.82$     & $6.92$     & $7.06$     & $7.21$   \\
$0.3$ & $6.85$     & $6.72$     & $6.84$     & $7.01$     & $7.17$     & $7.33$     & $7.54$     & $7.99$     & $8.83$   \\
$0.4$ & $6.88$     & $6.84$     & $7.02$     & $7.38$     & $7.60$     & $7.88$     & $8.35$     & $9.73$     & $13.26$   \\
$0.5$ & $6.91$     & $6.95$     & $7.22$     & $7.69$     & $8.03$     & $8.50$     & $9.73$     & $14.37$    & $88.06$   \\
$0.6$ & $6.95$     & $7.04$     & $7.45$     & $8.07$     & $8.61$     & $9.38$     & $12.58$    & $63.36$    & $\infty$   \\
$0.7$ & $6.99$     & $7.13$     & $7.77$     & $8.59$     & $9.42$     & $10.72$    & $21.97$    & $\infty$   & $\infty$   \\
$0.8$ & $7.03$     & $7.23$     & $8.22$     & $9.30$     & $10.67$    & $13.05$    & $\infty$   & $\infty$   & $\infty$   \\
$0.9$ & $7.03$     & $7.23$     & $8.22$     & $9.30$     & $10.67$    & $13.05$    & $\infty$   & $\infty$   & $\infty$  \\ \hline
&\multicolumn{9}{|c|}{$|core|$} \\ \hline
$\lambda$ & $10$ & $11$ & $12$ & $13$ & $14$ & $15$ & $16$ & $17$ & $18$\\ \hline
$0.1$ & $6.18$     & $6.24$     & $6.31$     & $6.35$     & $6.41$     & $6.58$     & $6.61$     & $6.83$   & $7.12$ \\
$0.2$ & $7.37$     & $7.52$     & $7.71$     & $8.43$     & $8.76$     & $10.41$    & $11.09$    & $11.81$  & $18.92$ \\
$0.3$ & $9.29$     & $9.86$     & $10.55$    & $15.01$    & $17.02$    & $44.86$    & $68.28$    & $136.24$ & $\infty$ \\
$0.4$ & $15.42$    & $18.64$    & $23.83$    & $\infty$   & $\infty$   & $\infty$   & $\infty$   & $\infty$ & $\infty$ \\
$0.5$ & $\infty$   & $\infty$   & $\infty$   & $\infty$   & $\infty$   & $\infty$   & $\infty$   & $\infty$ & $\infty$ \\
$0.6$ & $\infty$   & $\infty$   & $\infty$   & $\infty$   & $\infty$   & $\infty$   & $\infty$   & $\infty$ & $\infty$ \\
$0.7$ & $\infty$   & $\infty$   & $\infty$   & $\infty$   & $\infty$   & $\infty$   & $\infty$   & $\infty$ & $\infty$ \\
$0.8$ & $\infty$   & $\infty$   & $\infty$   & $\infty$   & $\infty$   & $\infty$   & $\infty$   & $\infty$ & $\infty$ \\
$0.9$ & $\infty$   & $\infty$   & $\infty$   & $\infty$   & $\infty$   & $\infty$   & $\infty$   & $\infty$ & $\infty$ \\ \hline
\end{tabular}
 \caption{The objective values for Example}\label{wexam-core}
 \end{table}
As the table gives, when the core length and the arrival rate values are raised
the delay time goes to infinity. If it is supposed the arrival rate is a fixed
value and the length of the core is increased, generally the objective value
increases and may goes to infinity. Similarly if it's supposed the length of the core
is a fixed value and the arrival rate is increased the objective value increases and
may goes to infinity.

\begin{table}
\centering
{
\begin{tabular}{|c||ccc||ccc||ccc|c|}
\hline
Number of&\multicolumn{3}{|c|}{ $\lambda=0.1$, $|core|$}&\multicolumn{3}{|c|}{$\lambda=0.4$,$|core|$}&\multicolumn{3}{|c|}{$\lambda=0.8$,$|core|$}& cpu time\\
vertices &4           & 10          &  16          & 4           &10             &16              & 4             & 10            & 16      & \\\hline
$20$     &4.13   &  4.13            &4.80          & 5.34        &$\infty$       &$\infty$        & 6.08          &$\infty$       &$\infty$  & 0.20 \\
$20$     &3.74   &  3.67            &4.14          &4.38         &11.44          &$\infty$        &5.40           &$\infty$       &$\infty$  & 0.19 \\
$20$     &4.83   &  4.81            &5.42          &5.80         &9.13           &$\infty$        &6.68           &$\infty$       &$\infty$  & 0.21\\ \hline
$30$     &6.95   &  6.93            &7.50          &8.91         &22.67          &$\infty$        &$\infty$       &$\infty$       &$\infty$  & 0.25 \\
$30$     &7.41   &  7.61            &8.52          &14.64        &$\infty$       &$\infty$        &$\infty$       &$\infty$       &$\infty$  & 0.26  \\
$30$     &6.87   &  8.01           &8.02         &14.09        &$\infty$       &$\infty$        &$\infty$       &$\infty$       &$\infty$    & 0.25  \\ \hline
$50$     &9.91   &  9.82            &10.28         &11.07        &$\infty$       &$\infty$        &11.57          &$\infty$       &$\infty$  & 0.76 \\
$50$     &8.21   &  8.11            &8.33          &9.11         &12.41          &$\infty$        &10.03          &$\infty$       &$\infty$  & 0.60  \\
$50$     &7.29   &  7.25            &7.70          &8.26         &10.47          &$\infty$        &10.28          &$\infty$       &$\infty$  & 0.54  \\ \hline
$90$     &11.96   &  11.83          &11.92         &12.69        &13.60          &14.52           &13.38          &19.22          &$\infty$  & 3.45  \\
$90$     &11.56   &  11.46          &11.63         &12.96        &12.63          &13.53           &12.38          &17.86          &$\infty$  & 4.59  \\
$90$     &12.92   &  12.88          &12.99         &13.56        &14.54          &$\infty$        &15.98          &44.24          &$\infty$  & 4.23   \\ \hline
$140$    &15.43  &   15.38          &15.55         &22.87        &$\infty$       &$\infty$        &$\infty$       &$\infty$       &$\infty$  & 16.48  \\
$140$    &16.48  &   16.74          &17.24         &23.92        &$\infty$       &$\infty$        &$\infty$       &$\infty$       &$\infty$  & 13.08  \\
$140$    &12.55  &   12.53          &12.69         &14.75        &16.65          &22.91           &$\infty$       &$\infty$       &$\infty$  & 14.94  \\ \hline
$200$    &15.30  &   15.15          &15.43         &16.43        &$\infty$       &$\infty$        &16.70          &$\infty$       &$\infty$  & 158.60 \\
$200$    &13.87  &   13.71          &8.87          &14.63        &14.93          &17.35           &13.85          &16.59          &$\infty$  & 144.92 \\
$200$    &12.69  &   12.60          &17.72         &13.63        &15.16          &17.58           &13.89          &$\infty$       &$\infty$  & 144.55 \\ \hline
$250$    &15.36  &   15.17          &15.24         &16.49        &$\infty$       &$\infty$        &16.74          &$\infty$       &$\infty$  & 231.69 \\
$250$    &14.46  &   14.72          &15.10         &21.93        &$\infty$       &$\infty$        &$\infty$       &$\infty$       &$\infty$  & 213.91\\
$250$    &12.50  &   12.39          &12.56         &13.48        &13.54          &33.46           &16.72          &18.97          &$\infty$  & 200.48\\ \hline
$500$    &18.49  &   18.52          &18.58         &20.74        &23.32          &27.87           &$\infty$       &$\infty$       &$\infty$  & 383.90\\
$500$    &20.62  &   20.64          &20.63         &20.72        &21.56          &21.64           &35.78          &42.39          &$\infty$  & 393.60\\
$500$    &21.31  &   22.14          &21.34         &28.80        &28.83          &$\infty$        &$\infty$       &$\infty$       &$\infty$  & 357.10 \\ \hline
$800$    &22.57  &   22.43          &23.62         &23.64        &24.90          &$\infty$        &27.71          &35.18          &$\infty$  & 443.36\\
$800$    &19.49  &   19.56          &19.82         &26.33        &$\infty$       &$\infty$        &$\infty$       &$\infty$       &$\infty$  & 452.18\\
$800$    &18.62  &   18.53          &18.96         &19.84        &21.69          &23.85           &21.39          &$\infty$       &$\infty$  &451.74 \\ \hline
$1000$   &20.03  &20.36             &20.48         &23.11        &25.65          &29.12           &31.71          &$\infty$       &$\infty$  &673.90\\
$1000$   &24.30  &24.45             &24.37         &28.23        &29.16          &34.03           &29.64          &38.97          &$\infty$  &684.65\\
$1000$   &20.53  &20.45             &21.81         &25.60        &27.32          &33.23           &$\infty$       &$\infty$       &$\infty$  &661.28\\ \hline
$1200$   &22.86  &22.73             &22.34         &23.07        &$\infty$       &$\infty$        &29.66          &$\infty$       &$\infty$  &785.61\\
$1200$   &25.36  &25.34             &25.81         &26.29        &27.93          &29.43           &30.51          &45.72          &$\infty$  &791.45\\
$1200$   &26.33  &26.52             &26.74         &26.82        &28.41          &31.12           &$\infty$       &$\infty$       &$\infty$  &796.37 \\ \hline
\end{tabular}}
\caption{The objective values for random generated examples}\label{wexam2-core}
 \end{table}

Next the algorithm run for 12 randomly generated problems.
The arcs lengths have been chosen from the set $\{1,...,5\}$
and the number of vertices have been chosen from the set $\{20,30,50,90,140,200,250,500,800,1000,1200\}$.
As the previous example the weights of vertices are all equal
and considered to be $w_i=\frac{1}{num}$, where $num$ is the number of vertices.

\section{Summary and conclusion}
In this paper, an stochastic queue core problem on the tree is investigated when the
the core is operated in an $M/G/1$ environment. First some properties for general core
on the tree are provided. Then the queue l-core problem is studied and an algorithm
having $\max\{o(n^2 l),o(n^2 log^2(n))\}$ time complexity is presented.

\end{document}